\documentclass{amsart}
\usepackage{amsfonts,amsthm,amssymb,color,soul}
\usepackage{setspace}
\usepackage{mathtools}
\usepackage{url}
\usepackage{tikz} 
\usepackage{enumerate}
\usepackage[all]{xy}

\newtheorem{theorem}{Theorem}[section]
\newtheorem{lemma}[theorem]{Lemma}
\newtheorem{prop}[theorem]{Proposition}
\newtheorem{cor}[theorem]{Corollary}

\theoremstyle{definition}
\newtheorem{definition}[theorem]{Definition}
\newtheorem{algorithm}[theorem]{Algorithm}

\newtheorem{example}[theorem]{Example}

\newcommand{\braces}[1]{\left\{#1\right\}}

\newcommand{\Z}{\mathbb{Z}}
\newcommand{\Q}{\mathbb{Q}}
\newcommand{\Ck}{\mathfrak{C}}
\newcommand{\Gc}{\mathcal{G}}
\newcommand{\Oc}{\mathcal{O}}
\newcommand{\Pc}{\mathcal{P}}
\newcommand{\pk}{\mathfrak{p}}
\newcommand{\qk}{\mathfrak{q}}
\newcommand{\uk}{\mathfrak{u}}

\newcommand{\GcO}{\mathcal{G}_1}
\newcommand{\GcEE}{\mathcal{G}_2}
\newcommand{\GcEO}{\mathcal{G}_3}

\numberwithin{theorem}{section}

\title{On primitive Pythagorean triples of special forms}
\author{Andrew Schmelzer}
\email{aschmelze001@csbsju.edu}
\author{Sunil Chetty}
\email{schetty@csbsju.edu}
\address{Mathematics Department, College of Saint Benedict and Saint John's University, St. Joseph, MN 56374}

\thanks{We thank the College of Saint Benedict and Saint John's University for supporting this research during Summer 2019. We thank Brandon Alberts for his helpful conversation ans suggestions regarding the density results. Computational exploration that guided this research was conducted using SageMath \cite{Sage}.}
\date{\today}

\begin{document}

\begin{abstract}
	We explore primitive Pythagorean triples of special forms $(a,b,b+g)$ and $(a,a+f,c)$, with $g,f\in\Z^+$. For each $g$ and $f$, we provide a method to generate infinitely many such primitive triples. Lastly, for each $g$, we describe the asymptotic density of primitive $(a,b,b+g)$ triples within all primitive triples of the same parity.
\end{abstract}

\maketitle
%
%
%
%
%
%
	
	\section{Introduction}
	\label{intro}
		A \emph{Pythagorean Triple} (PT) is a triple of integers $(a,b,c)\in\mathbb{Z}^3$ such that $a^2+b^2=c^2$. A Pythagorean triple is a \emph{primitive Pythagorean triple} (PPT) if $d>0$ and $d$ divides $a,b,$ and $c$ implies $d=1$. To see a Pythagorean triple is primitive, it suffices to show $d>0$ dividing any two of $a$, $b$, or $c$ implies $d=1$.\\
		\indent The two theorems below are classical results, and we include a proof of the second for completeness sake and to demonstrate the spirit of several arguments to follow in \S\ref{PPTg}.
	
	\begin{theorem}
	\label{PToldthm}
		All Pythagorean triples $(a,b,c)$ take the form $a=r^2-s^2$, $b=2rs$, $c=r^2+s^2$ (up to switching the parity for $a$ and $b$) for some $r,s\in\mathbb{Z}$ with $s<r$.
	\end{theorem}
	\begin{proof}
		See, for example, \cite[\S VII.2]{Davenport}
	\end{proof}
	\begin{theorem}
	\label{PPToldthm}
		In the notation of Theorem \ref{PToldthm}, a Pythagorean triple is primitive if and only if $\gcd(r,s)=1$.
	\end{theorem}
	\begin{proof}
		If $(a,b,c)$ is a PPT and $\gcd(r,s)=d$ then $d$ is also a factor of $a=r^2-s^2$ and $b=2rs$, so $d=1$.\\
		\indent Now suppose $\gcd(r,s)=1$. The parity of $r$ and $s$ must be opposite, otherwise $a$ and $b$ would both be even. So we take $r$ odd and $s$ even. If $d$ is a common factor of $a=r^2-s^2$, $b=2rs$, and $c=r^2+s^2$, then $d$ is odd because $a$ is odd. Also, $d$ divides the sum and difference of $c$ and $a$, i.e. $2s^2=c-a$ and $2r^2=c+a$, and hence $d$ divides $\gcd(r,s)$, so $d\mid 1$.
	\end{proof}

	Because there are infinitely many pairs $r,s\in\Z$ with $\gcd(r,s)=1$, it follows from Theorem \ref{PPToldthm} that there are infinitely many primitive Pythagorean triples. However, the result does not immediately provide information about families of primitive triples that share a particular quality.
	\section{Primitive triples of the form $(a,b,b+g)$}
	\label{PPTg}
	Using only elementary algebra, one can show that there are infinitely many triples of the form $(a,b,b+1)$. Extending this, one can ask if there are infinitely many PPT of the form $(a,b,b+g)$, with $g>1$.
	\subsection{Conditions on $g$}
	\label{gconditions}
	\begin{lemma}
	\label{PTparity}
		In a PPT, exactly one of $a,b$ is odd, and $c$ is always odd.
	\end{lemma}
	\begin{proof}
		If both $a$ and $b$ are even, then $c$ is also even and $(a,b,c)$ is not primitive. If $a$ and $b$ are both odd, then $a^2\equiv b^2\equiv 1\mod{4}$. Thus $c^2\equiv 2\mod{4}$, but 2 is not a square modulo 4. Thus, with all other cases exhausted, we have that exactly one of $a,b$ is odd.
	\end{proof}
%
%
%
	The following proposition gives the form of $g$ in primitive $(a,b,b+g)$ triples.	
	\begin{prop}
	\label{Conj1}
		If $(a,b,b+g)$ is a PPT then $g=m^2$, with $m$ odd, or $g=2m^2$, with $m$ any positive integer.
	\end{prop}
	\begin{proof}
		Let $(a,b,b+g)$ be a PPT, and suppose $g$ is odd. Because of Lemma \ref{PTparity}, it must be $b=2rs$, otherwise $c=b+g$ would be even and the triple would not be primitive. Then by Theorem \ref{PToldthm} our triple fits the form $a=r^2-s^2$, $b=2rs$, and $b+g=r^2+s^2$. Solving for $g$ yields $g=r^2+s^2-2rs=(r-s)^2$, so for $m=r-s$ we have $g=m^2$.\\
		\indent Suppose $m$ is even. Note that $a=(r-s)(r+s)$. If $r-s$ is even, then $r+s$ is even too, and so $a$ would be even. Then both of $a,b$ would be even, and that contradicts Lemma \ref{PTparity}, so $m$ must be odd.\\
		\indent Now, suppose $g$ is even. Then, again using Theorem \ref{PToldthm}, our triple fits the form $a=2rs$, $b=r^2-s^2$, and $b+g=r^2+s^2$. As before, we use Lemma \ref{PTparity} to know $b=r^2-s^2$, otherwise $b$ and $c=b+g$ are even. Solving for $g$ yields $g=2s^2$, so for $m=s$ and we have $g=2m^2$.
	\end{proof}	
	\subsection{Generating primitive $(a,b,b+g)$ triples}
	\label{genPPTg}
	Now that we know the shape of all possible $g$, we discuss generating the triples.	
	\begin{lemma}
	\label{bfromag}
		Any PPT of the form $(a,b,b+g)$ is determined by $a$ and $g$.
	\end{lemma}	
	\begin{proof}
		We know that $a^2+b^2=(b+g)^2=b^2+2bg+g^2$. Solving for $b$ yields $$b=\frac{a^2-g^2}{2g}.$$
	\end{proof}
	\begin{theorem}
	\label{Conj2}
		For each $g\in \{(2k+1)^2:k\in\mathbb{Z}\}\cup\{2(m)^2:m\in\mathbb{Z}\}$, there are an infinite number of PPTs of the form $(a,b,b+g)$.
	\end{theorem}	
	\begin{proof}
		For each case of $g$, we choose values $t$ and $q$, and construct a family $(a_n)=(tn+q)$ parametrized by $n\in\Z^+$ in order to define an infinitely family of PT $(a_n,b_n,b_n+g)$.\\
		
		Let $g=m^2$ for some odd $m$. Let $t=2m$, $q=m$, so $a_n=2mn+m=m(2n+1)$. Then, applying Lemma \ref{bfromag},
			$$b_n=\frac{m^2(2n+1)^2-m^4}{2m^2}=\frac{(2n+1)^2-m^2}{2}=\frac{(2n+1-m)(2n+1+m)}{2}.$$
		Note that in the right-most expression both elements of the numerator are even, making $b_n$ even. Then our triple is
			$$\left( m(2n+1),\frac{(2n+1)^2-m^2}{2},\frac{(2n+1)^2+m^2}{2} \right).$$
		Arguing as in Theorem \ref{PPToldthm}, this form generates a PPT whenever $\gcd(m,2n+1)=1$ There are an infinite number of primes of the form $2n+1$ which do not divide $m$, so that condition holds an infinite number of times.\\
		
		Now, let $g=2m^2$ for some odd $m$. Let $t=2m$, $q=0$, so $a_n=tn=2mn$, and then
			$$b=\frac{a^2-g^2}{2g}=\frac{4m^2n^2-4m^4}{4m^2}=n^2-m^2.$$
		Then we have $\left(2mn, n^2-m^2, n^2 + m^2 \right)$ is a PPT whenever $\gcd(n,m)=1$, exactly by Theorem \ref{PPToldthm}. There are an infinite number of primes $n$ not dividing $m$, so that condition holds an infinite number of times.\\
		
		Next, let $g=2m^2$ for some even $m$. Let $t=4m$ and $q=2m$, so $a_n=tn+q=2m(2n+1).$ Then
			$$b_n=\frac{a^2-g^2}{2g}=\frac{4m^2(2n+1)^2-4m^4}{4m^2}=(2n+1)^2-m^2.$$
		Then we have 
			$$\left(2m(2n+1), (2n+1)^2-m^2, (2n+1)^2 + m^2 \right).$$
		This form generates a PPT whenever $\gcd(2n+1,m)=1$, by Theorem \ref{PPToldthm} with $r=2n+1$ and $s=m$. There are an infinite number of primes of the form $2n+1$ not dividing $m$, so that condition holds an infinite number of times.\\
		
		Thus, for every $g\in \{(2m+1)^2:m\in\mathbb{Z}\}\cup\{2(m)^2:m\in\mathbb{Z}\}$, there are an infinite number of PPTs of the form $(a,b,b+g)$. 
	\end{proof}
	Note the above construction of $(a_n,b_n,b_n+g)$ provides a family of PT with strictly increasing sequence of first coordinates $(a_n)_g$, for each $g$ as in Theorem \ref{Conj2}. In turn, we have a subfamily of PPTs with strictly increasing first coordinates. 
	\begin{cor}
	\label{CorConj2}
		Every PPT of the form $(a,b,b+g)$, is contained in one of the families $(a_n,b_n,b_n+g)$. Moreover, $(a,b,b+g)$ is obtained from $r_n$ and $s_n$ (as in Theorem \ref{PPToldthm}) according to:
			\begin{center}
			\begin{tabular}{l|l|l}
				$g=m^2$ & $g=2m^2$ & $g=2m^2$\\
				\hline
				$m$ odd & $m$ odd & $m$ even\vspace{-2mm}\\
				 & & \\
				$r_n=\frac{1}{2}\left(2n+1+m\right)$ & $r_n=n$ & $r_n=2n+1$\\
				$s_n=\frac{1}{2}\left(2n+1-m\right)$ & $s_n=m$ & $s_n=m$\vspace{-2mm}\\
				 & & \\
				$\gcd(2n+1,m)=1$ & $\gcd(n,m)=1$ & $\gcd(2n+1,m)=1$\\
			\end{tabular}
			\end{center}
	\end{cor}
	\begin{proof}
		For $g$ odd, we have $m$ odd. The proof of Proposition \ref{Conj1} shows $m=r-s$, $a=(r-s)(r+s)$, and both are odd. So $a=mk$ for some odd $k$. For $g$ even, Proposition \ref{Conj1} shows $m=s$ and $a=2rs$, so $a=mk$, and both cases of $k$ even and $k$ odd arise as $a_n$.
	\end{proof}
	\section{Primitive triples of the form $(a,a+f,c)$}
	\label{PPTf}
	
	We now discuss a second form of PPT, with a fixed difference $f=b-a$. First, we look into the form of $f$.
	\subsection{Conditions on $f$}
	\label{fconditions}
	\begin{prop}
	\label{fprimes}
		Suppose $p$ is a prime dividing $f$. If $(a,a+f,c)$ is a PPT, then $p\equiv \pm 1$ (mod $8$).
	\end{prop}
	\begin{proof}
		Let $p$ be a prime dividing $f$ and consider the defining equation modulo $p$, i.e. 
		$$c^2\equiv a^2+(a+f)^2\equiv 2a^2\mod{p}.$$
		Notice that $p$ does not divide $a$, otherwise $p$ also divides $b$ and $c$, and so $(a,a+f,c)$ is not primitive. Now, $(ca^{-1})^2\equiv 2\mod{p}$, making 2 a quadratic residue, and by quadratic reciprocity (see \cite[\S III.5]{Davenport} or \cite[\S 5.2]{IR}) we have $p\equiv \pm 1\mod{8}$.
	\end{proof}
	
	Before continuing, we recast the defining equation for the PPT in question. Let $(a,a+f,c)$ be a PPT. Then, by elementary algebra, we see
		$$\begin{array}{rrl}
			a^2+(a+f)^2=c^2 & \Leftrightarrow & 2a^2+2af+f^2=c^2\\
				& \Leftrightarrow & 4a^2+4af+2f^2-2c^2=0\\
				& \Leftrightarrow & (2a+f)^2+f^2-2c^2=0\\
				& \Leftrightarrow & (2a+f)^2-2c^2=-f^2,
		\end{array}$$
	with the last in the form of a Pell equation $x^2-2y^2=-f^2$.\\
	\indent Pell equations are well-studied Diophantine equations, and expanding our view to the quadratic ring of integers $\Z[\sqrt{2}]$ is a particularly fruitful strategy. We now record some useful information about $\Z[\sqrt{2}]$. First, we define the norm function $N$ on $\Z[\sqrt{2}]$ by $N(\beta)=\beta\Bar{\beta}$.
	
	\begin{theorem}
	\label{Z2ED}
		$\Z[\sqrt{2}]$ is a Euclidean domain. More specifically, if $\alpha,\beta\in\Z[\sqrt{2}]$ and $\beta\neq 0$ then there exist unique $\gamma,\rho\in\Z[\sqrt{2}]$ such that $\alpha=\beta\gamma+\rho$ with $|N(\rho)|<|N(\beta)|$.
	\end{theorem}
	\begin{proof}
		We employ a classical argument for quadratic rings, as in \cite[\S 1.4]{IR} for $\Z[i]$. Let $\alpha\in\Z[\sqrt{2}]$ and $0\neq\beta\in\Z[\sqrt{2}]$. From $\frac{\alpha}{\beta}=\frac{\alpha\bar{\beta}}{\beta\bar{\beta}},$ and both $\alpha\bar{\beta},\beta\bar{\beta}\in\Z[\sqrt{2}]$, we see $\frac{\alpha}{\beta}\in\Q(\sqrt{2})$. Hence $\frac{\alpha}{\beta}=r+s\sqrt{2}$, with $r,s\in\Q$. Let $u,v\in\Z$ such that $|r-u|\leq\frac{1}{2}$ and $|s-v|\leq\frac{1}{2}$, and define $\gamma=u+v\sqrt{2}$ and $\rho=\alpha-\beta\gamma$. By the multiplicativity of the the norm map $N$, 
			$$|N(\rho)|=\left|N\left(\frac{\alpha}{\beta}-\gamma\right)N(\beta)\right|=\left|N\left(\frac{\alpha}{\beta}-\gamma\right)\right||N(\beta)|$$
		with
			$$\left|N\left(\frac{\alpha}{\beta}-\gamma\right)\right|=\left|(r-u)^2-2(s-v)^2\right|.$$
		Since $|r-s|\leq\frac{1}{2}$ and $|s-v|\leq\frac{1}{2}$, we have $0\leq (r-u)^2\leq \frac{1}{4}$ and $0\leq (s-v)^2\leq \frac{1}{4}$, and in turn $\left|N\left(\frac{\alpha}{\beta}-\gamma\right)\right|<1$. We conclude $|N(\rho)|<|N(\beta)|$ as desired.\\
		\indent The uniqueness of $\gamma$ and $\rho$ follow from the condition that $|r-u|,|s-v|\leq\frac{1}{2}.$
	\end{proof}
	
	\begin{cor}
	\label{Z2UFD}
		If $\pk\subset\Z[\sqrt{2}]$ is an ideal, then $\pk=(\alpha)$ for some $\alpha\in\Z[\sqrt{2}]$. Moreover, every non-unit $\alpha\in\Z[\sqrt{2}]$ factors into primes uniquely.
	\end{cor}
	\begin{proof}
		Both parts of the claim are applications of classical results (e.g. \cite[Theorems 18.3-18.4]{Gallian}) in the theory of rings, following Theorem \ref{Z2ED}.
	\end{proof}
	
	In addition to the above broad structural results, we recall the following well-known factorization of particular ideals in $\Z[\sqrt{2}]$.
	
	\begin{theorem}
	\label{fprimefactors}
		If $p\in\Z$ is a prime such that $p\equiv \pm 1\mod{8}$, then $(p)=\pk\bar{\pk}$, with $\pk,\bar{\pk}$ prime ideals in $\Z[\sqrt{2}]$.
	\end{theorem}
	\begin{proof}
		See, for example, \cite[Theorem 13.1.3]{IR}.
	\end{proof}
	
%
%
	\subsection{Generating $(a,a+f,c)$ triples}
	\label{genPPTf}
	Let $\gamma=1+\sqrt{2},$ and $\delta=\gamma^2=3+2\sqrt{2}$.
	
	\begin{theorem}
	\label{Pellneg1}
		All solutions of $x^2-2y^2=-1$ are given by $\pm\gamma\delta^m$, for $m\in\mathbb{Z}$.
	\end{theorem}
	\begin{proof}
		See, for example, \cite[\S 8.2]{LeVeque}.
	\end{proof}
	
	We see a recurrence within solutions of the classical negative Pell equation in the previous theorem.
	
	\begin{prop}
	\label{Pellrecursion}
		Suppose $\delta=3+2\sqrt{2}$ is as above, and $t=j+k\sqrt{2}$. Define, for $n\geq 0$, sequences $A_n$ and $B_n$ by $t\delta^n=A_nt+B_n(2k+j\sqrt{2})$. Then
		\begin{center}
			\begin{tabular}{llll}
				$A_0=1$,\hspace{5mm} & $A_1=3$,\hspace{5mm} & and\hspace{5mm} & $A_n=6A_{n-1}-A_{n-2}$, for $n\geq 2$,\\
				$B_0=0$, & $B_1=2$, & and & $B_n=6B_{n-1}-B_{n-2}$, for $n\geq 2$.\\
			\end{tabular}
		\end{center}
	\end{prop}
	\begin{proof}
		We first establish the initial conditions and then intermediate recurrence relations for $A_n$ and $B_n$. The case $n=0$ is clear as $t\delta^0=t$. Next,
		$$t\delta=3(j+k\sqrt{2})+2\sqrt{2}(j+k\sqrt{2})=3t+2(2k+j\sqrt{2}),$$
		showing that $A_1=3$ and $B_1=2$.\\
		
		Now, assume $A_k$ and $B_k$ satisfy the given defining equation for all $k<n$. We compute
		$$\begin{array}{rl}
		t\delta^n= & \left(A_{n-1}(j+k\sqrt{2})+B_{n-1}(2k+j\sqrt{2})\right)(3+2\sqrt{2})\\
		= & 3A_{n-1}j+3A_{n-1}k\sqrt{2}+4A_{n-1}k+4B_{n-1}j+6B_{n-1}k\\
		& \hspace{4mm}+2A_{n-1}j\sqrt{2}+3B_{n-1}j\sqrt{2}+4B_{n-1}k\sqrt{2}\\
		= & (3A_{n-1}+4B_{n-1})j+(3A_{n-1}+4B_{n-1})k\sqrt{2}\\
		& \hspace{4mm}+(4A_{n-1}+6B_{n-1})k+(2A_{n-1}+3B_{n-1})j\sqrt{2}\\
		= & (3A_{n-1}+4B_{n-1})t+(2A_{n-1}+3B_{n-1})(2k+j\sqrt{2}).\\
		\end{array}$$
		Thus, we see $A_n=3A_{n-1}+4B_{n-1}$ and $B_n=2A_{n-1}+3B_{n-1}$. To finish, we show that these same mixed recurrence relations are satisfied when $A_n=6A_{n-1}-A_{n-2}$ and $B_n=6B_{n-1}-B_{n-2}$.\\
		
		Consider $x_n$ and $y_n$ defined by the two recurrence relations of the conclusion of the claim. Suppose the mixed recurrence relations on $A_n$ and $B_n$ from the previous paragraph hold for $x_k$ and $y_k$, for every $k<n$. We compute
		$$\begin{array}{rlrl}
		x_n= & 6(3x_{n-2}+4y_{n-2})-x_{n-2} & y_n= & 6(2x_{n-2}+3y_{n-2})-y_{n-2}\\
		= & 3x_{n-1}+3(3x_{n-2}+4y_{n-2})-x_{n-2} & = & 3y_{n-1}+3(2x_{n-2}+3y_{n-2})-y_{n-2}\\
		= & 3x_{n-1}+8x_{n-2}+12y_{n-2} & = & 3y_{n-1}+8y_{n-2}+6x_{n-2}\\
		= & 3x_{n-1}+4(2x_{n-2}+3y_{n-2}) & = & 3x_{n-1}+2(3x_{n-2}+4y_{n-2})\\
		= & 3x_{n-1}+4y_{n-1} & = & 3y_{n-1}+2x_{n-1}.\\
		\end{array}$$
		Thus, when satisfying the same initial conditions, we see $A_n=x_n$ and $B_n=y_n$, for all $n\geq 0$.
	\end{proof}
	
	Let $(a,a+f,c)$ be a PPT. Recall that $a^2+(a+f)^2=c^2$ if and only if $(2a+f)^2-2c^2=-f^2$. Now suppose $f=p_1^{t_1}\ldots p_n^{t_n}$, with $p_i\in \Z, t_i\in\Z^+$. By Proposition \ref{fprimes}, we know that all the primes $p_i$ dividing $f$ satisfy $p_i\equiv\pm 1 \mod 8$. Since every ideal $(p_i)$ splits over $\Z[\sqrt{2}]$ according to Theorem \ref{fprimefactors}, we know that the ideal $(f)$ can also be split into conjugate pair factors over $\mathbb{Z}[\sqrt{2}]$.\\
	\indent Now, let $(f)=\pk_1^{t_1}\bar{\pk}_1^{t_1}\ldots \pk_n^{t_n}\bar{\pk}_n^{t_n}$ and 
		$$\Ck_f:=\braces{\prod_i \qk_i^{t_i}:~\qk_i=\pk_i~\text{or}~\qk_i=\bar{\pk}_i,~\text{for each}~1\leq i\leq n}.$$
	Note that if $\uk\in \Ck_f$, then whenever $\qk_i\mid\uk$, we have $p_i\nmid \uk$. Also $\uk\in\Ck_f$ implies $\uk\bar{\uk}=f$, and hence $N(\uk)=f$.
	
	\begin{theorem}
	\label{Conj3}
		Recall $\gamma=1+\sqrt{2},~\delta=\gamma^2\in\Z[\sqrt{2}]$.	All the PPTs of the form $(a,a+f,c)$ are also of the form $a=\frac{X-f}{2}$, $b=\frac{X+f}{2}$, and $c=Y$, where $X,Y$ are the components of $\pm\gamma\delta^m\uk^2$, for some $m\in \mathbb{Z}$ and some $\uk\in\Ck_f$.
	\end{theorem}
	
	\begin{proof}
		Let $(a,a+f,c)$ be a PPT, so $(2a+f)^2-2c^2=-f^2$. Let $\uk\in\Ck_f$, so we have
			$$((2a+f)-\sqrt{2}c)((2a+f)+\sqrt{2}c)=-\uk^2\bar{\uk}^2.$$
		Without loss of generality, let $\pk_i\mid \uk$ and $\bar{\pk}_i\mid \bar{\uk}$. Suppose that both $\pk_i$ and $\bar{\pk}_i$ divide $(2a+f+\sqrt{2}c)$. Then $\pk_i\bar{\pk}_i \mid (2a+f+\sqrt{2}c)$, i.e. $p_i\mid (2a+f+\sqrt{2}c)$, and since $p_i\in\Z$, we know $p_i$ divides $c$.\\
		\indent We also know that $p_i\mid f$, so $p_i$ divides two terms of $(2a+f+\sqrt{2}c)$, hence also must divide $2a$. From $p_i\mid f$, we also know $p_i\equiv\pm1\mod8$, and hence $p_i$ does not divide $2$. So $p_i\mid a$.\\
		\indent Now, we have that $p_i\mid a,a+f,c$, contradicting that $(a,a+f,c)$ is a PPT. Thus, for any prime $\pk_i$ dividing $(2a+f+\sqrt{2}c),$ we know that $\bar{\pk}_i\nmid (2a+f+\sqrt{2}c)$. Since no prime and it's conjugate divide $\uk$, it must be that $\uk^2\mid (2a+f+\sqrt{2}c)$.\\
		\indent Let $\alpha=2a+f+\sqrt{2}c$. From $N(\uk^2)=f^2$ and $N(\alpha)=-f^2$, we now have $N\left(\frac{\alpha}{\uk^2}\right)=-1$. Then, by Theorem \ref{Pellneg1}, the components $x_0+\sqrt{2}y_0$ of $\frac{\alpha}{\uk^2}$ are solutions to $x^2-2y^2=-1$. So $\frac{\alpha}{\uk^2}=\pm\gamma\delta^m$, for some $m$. Then
			$$\alpha=2a+f+ \sqrt{2}c =\pm\gamma\delta^m \uk^2,$$
		and in turn $X=2a+f,Y=c.$ So $a=\frac{X-f}{2}$, $c=Y$.
	\end{proof}
	
	The above theorem allows us to produce infinitely many triples of the desired form from solutions to the Pell equation $x^2-2y^2=-1$. We now restrict\footnote{The restriction on $f$ is for convenience and clarity of notation. Both Algorithm \ref{algoPPTf} and Theorem \ref{verifyPPTf} can be generalized to any $f$ satisfying the conditions of \S\ref{fconditions}.} to the case $f\equiv \pm 1\mod{8}$ is prime, and combine Proposition \ref{Pellrecursion} and Theorem \ref{Conj3} as an algorithm.
	
	\begin{algorithm}
	\label{algoPPTf}
		Recall $\gamma=1+\sqrt{2},~\delta=\gamma^2\in\Z[\sqrt{2}]$. Fix a rational prime $f$ with $f\equiv \pm 1\mod{8}$. A PPT $(a,a+f,c)$ can be generated\footnote{Each step can be accomplished with functionality available in Sage \cite{Sage}} as follows:
			\begin{enumerate}
				\item Factor the ideal $(f)$ over $\Z[\sqrt{2}]$ as $(f)=\pk\bar{\pk}$.
				\item Determine a generator $\uk\in\Z[\sqrt{2}]$ of $\pk$.
				\item Choose $m\in\Z^+$. Set $A_0=1$, $A_1=3$, $B_0=0$, and $B_1=2$.\\
						For $2\leq i\leq m$, compute $A_i=6A_{i-1}-A_{i-2}$ and $B_i=6B_{i-1}-B_{i-2}$
				\item Set $X+Y\sqrt{2}=((A_m+2B_m)+(A_m+B_m)\sqrt{2})\uk^2$
				\item Set $a=\frac{X-f}{2}$, $b=\frac{X+f}{2}$, and $c=Y$.
			\end{enumerate}
	\end{algorithm}
	\begin{theorem}
	\label{verifyPPTf}
		Let $f,~\gamma,~\delta$, and $\uk$ be as in Algorithm \ref{algoPPTf}. If $(a,b,c)$ is generated by Algorithm \ref{algoPPTf} then it is a PPT.
	\end{theorem}
	\begin{proof}
		That $(a,b,c)$ is a PT is established in the paragraph preceding Theorem \ref{Z2ED}, along with Theorem \ref{Pellneg1} and Proposition \ref{Pellrecursion}.\\
		\indent Suppose $(a,b,c)$ is not primitive, and consider $q\in\Z$ with $q\mid a$, $q\mid b$, and $q\mid c$. As $b-a=f$ and $b+a=X$, we have $q=f$ and so $f\mid X$ and $f\mid Y$. Now $f\beta=\gamma\delta^m\uk^2$, for some $\beta\in\Z[\sqrt{2}]$, and from $f=\uk\bar{\uk}$, it follows $\bar{\uk}\beta=\gamma\delta^m\uk$. As $\gamma$ is a unit in $\Z[\sqrt{2}]$, this implies $\uk\in(\bar{\uk})$, i.e. $\pk\subset\bar{\pk}$. But this is impossible for $q=f\equiv \pm1\mod{8}$, by Theorem \ref{fprimefactors}.
	\end{proof}
	\section{Density of primitive $(a,b,b+g)$ triples}
	\label{densityPPTg}
	We aim to study the asymptotic density of the set of PPT of $(a,b,b+g)$ form among the set of all PPT with matching parity of $a$. In light of Theorem \ref{PPToldthm}, the set of all PPT of $(r^2-s^2,2rs,r^2+s^2)$ form (and separately PPT with $a=2rs$) is in bijection with the set of pairs $(r,s)$ with $s<r$ and $\gcd(r,s)=1$. We consider\footnote{Our approach is different from \cite[\S 13]{RTB}, who discusses all PPT with bounded hypotenuse, without restricting the parity of $a$ or $b$, and proves an asymptotic formula of Lehmer.}
		$$\Pc(B)=\braces{(r,s):~\gcd(r,s)=1,~0<s<r,~r\leq B}\subset\Z^2.$$
	\indent We begin by recalling some classical number theoretic functions (see for example \cite[\S 2]{Apostol}). The Euler totient function $\phi(n)$ is the number of $k$ such that $1\leq k\leq n$ and $\gcd(k,n)=1$. Also, the M\"{o}bius function $\mu(n)$ is defined by $\mu(1)=1$, $\mu(n)=(-1)^k$ for $n=p_1\cdots p_k$ a product of distinct primes, and $\mu(n)=0$ if $n$ has a square factor. 
	\begin{prop}
		\label{asymptoticPPT}
		$\#\Pc(B)=\frac{3}{\pi^2}B^2+\Oc(B\log B).$
	\end{prop}
	\begin{proof}
		By definition of $\Pc(B)$,
		$$\#\Pc(B)=\sum_{r\leq B}\#\braces{(r,s):~\gcd(r,s)=1,~0<s<r}=\sum_{r\leq B}\phi(r).$$
		By \cite[Theorem 3.7]{Apostol}, $\sum_{r\leq B}\phi(r)=\frac{3}{\pi^2}B^2+\Oc(B\log B)$.
	\end{proof}
	\indent In light of the proof of Theorem \ref{Conj2} and Corollary \ref{CorConj2}, we consider separately the cases $g=m^2$, with $m$ odd; $g=2m^2$, with $m$ even; and $g=2m^2$, with $m$ odd. By the work in Proposition \ref{Conj1}, these cases deal with PPT of the form $(r^2-s^2,2rs,2rs+m^2)$, $(2rs,r^2-s^2, r^2-s^2+2m^2)$, and $(2rs,r^2-s^2, r^2-s^2+2m^2)$, respectively. Thus, PPT in the first case will be compared with all PPT for which $b=2rs$, and the latter two cases will be compared with all PPT for which $a=2rs$.
	\begin{example}
		Let $g=1$. By Corollary \ref{CorConj2}, a PPT $(a,b,b+1)$ is obtained from $r_n=n+1$ and $s_n=n$, for some $n$. We let $\Gc(B)=\braces{(n+1,n):~n+1\leq B}$, and compute
		$$\lim_{B\to\infty}\frac{\#\Gc(B)}{\#\Pc(B)}=\lim_{B\to\infty}\frac{B}{\frac{3}{\pi^2}B^2}=0.$$
	\end{example}
	The example illustrates that PPT of the form $(a,b,b+g)$, for a fixed $g$, may be rare among all PPT, and so instead we consider the collection of all PPT of the form $(a,b,b+g)$, for each family of values of $g$ in Corollary \ref{CorConj2}.\\
	
	\subsection{Case $g=m^2$ odd} In this case, we know $m$ is odd. By Corollary \ref{CorConj2}, the PPT for this case are given by $r=\frac{1}{2}(k+m)$ and $s=\frac{1}{2}(k-m)$ for some odd integer $k$, with $\gcd(k,m)=1$. The set of such pairs $(k,m)$ is in bijection with the set of $(r,s)$ described by the formulas and satisfying $\gcd(r,s)=1$. Notice $m<k\leq B$ implies $r\leq B$. Let
		$$\GcO(B)=\braces{(k,m):~\gcd(k,m)=1,~0<m<k\leq B,~k,m\equiv 1\mod{2}}\subset\Z^2.$$
	\begin{prop}
		\label{equalparity}
		For $N\in\Z^+$ odd, there are $\phi(N)/2$ odd integers $1\leq m<N$ such that $\gcd(m,N)=1$.
	\end{prop}
	\begin{proof}
		For each odd $m$ with $1\leq m<N$, with $\gcd(m,N)=1$, we know $N-m$ is even, $1<N-m\leq N-1$, and $\gcd(N-m,N)=1$.
	\end{proof}
	We define the parity function $d_2$ by $d_2(n)=1$ if $2\nmid n$ and $d_2(n)=0$ if $2\mid n$.
	\begin{definition}
		\label{2Eulerdef}
		The 2-Euler totient function $\phi_2(n)$ is the product $\phi_2(n)=\left(\phi \cdot d_2\right)(n)$ of the Euler totient function $\phi$ and the parity function $d_2$, so
		$$\phi_2(n)=\left\{\begin{array}{ll} \phi(n) & \text{~if~} 2\nmid n\\ 0 & \text{~if~} 2\mid n.\end{array}\right.$$
	\end{definition}
	Because $\phi$ and $d_2$ are multiplicative functions, we know $\phi_2$ is a multiplicative function also. We set $f(n)=\sum_{d\mid n}\phi_2(d)$ and apply the M\"{o}bius inversion formula (see \cite[Theorem 2.9]{Apostol}) to obtain
	$$\phi_2(n)=\sum_{d\mid n}\mu(d)f\left(\frac{n}{d}\right).$$
	\begin{lemma}
		\label{2Eulerlemma1}
		For $n=2^e u$, with $e\geq 0$ and $2\nmid u$, $f(n)=u$. In particular, if $e=0$ then $f(n)=n$.
	\end{lemma}
	\begin{proof}
		All divisors $d$ of $2^e u$ are of the form $d=2^a v$ with $0\leq a\leq e$ and $v\mid u$. When $e=0$, we also have $a=0$ and so all factors of $n$ are odd. Thus
			$$f(n)=\sum_{d\mid n}\phi_2(d)=\sum_{d\mid n}\phi(d)=n,$$ 
		with the right-most equality by \cite[Theorem 2.2]{Apostol}.\\
		\indent For $e>0$, we know $\phi_2(d)=0$ when $2\mid d$, and so $\sum_{d\mid n}\phi_2(n)=\sum_{\text{odd}~d\mid n}\phi(d)$. The odd divisors of $n$ are the numbers $v$ dividing $u$. Now
			$$f(n)=\sum_{d\mid n}\phi_2(d)=\sum_{\text{odd}~d\mid n}\phi(d)=\sum_{v\mid u}\phi(v)=u,$$
		again using \cite[Theorem 2.2]{Apostol} for the last equality.\\
	\end{proof}
	\begin{theorem}
		\label{2Eulertheorem}
		$\sum_{n\leq B}\phi_2(n)=\frac{2}{\pi^2}B^2+\Oc(B\log_2 B).$
	\end{theorem}
	\begin{proof}
		The argument follows the same strategy as \cite[Theorem 3.7]{Apostol}. By the M\"{o}bius inversion formula stated above, we have
			$$\begin{array}{rl}
				\sum_{n\leq B}\phi_2(n) = & \sum_{n\leq B}\sum_{d\mid n} \mu(d)f\left(\frac{n}{d}\right)\vspace{2mm}\\
				= & \sum_{q,d,~qd\leq B}\mu(d)f(q)\vspace{2mm}\\
				= & \sum_{d\leq B}\mu(d)\sum_{q\leq B/d}f(q)\vspace{2mm}\\
				= & \sum_{d\leq B}\mu(d)\left(\sum_{\text{odd}~q\leq B/d}f(q)+\sum_{\text{even}~q\leq B/d}f(q)\right).\vspace{2mm}
			  \end{array}.$$
		By Lemma \ref{2Eulerlemma1}, and that the sum of the first $N$ consecutive odds yields $N^2$,
			$$\sum_{\text{odd}~q\leq B/d}f(q)=\sum_{\text{odd}~q\leq B/d} q=\left(\frac{B}{2d}\right)^2+\Oc(1)=\frac{B^2}{4d^2}+\Oc(1).$$
		For even $q\leq B/d$ we write $q=2^a j$, with $2\nmid j$, and so using Lemma \ref{2Eulerlemma1},
			$$\begin{array}{rl}
				\sum_{\text{even}~q\leq B/d}f(q) = & \sum_{a=1}^{M_B}\sum_{\text{odd}~j\leq B/(2^a d)} f(2^a j)\vspace{2mm}\\
				= & \sum_{a=1}^{M_B}\sum_{\text{odd}~j\leq B/(2^a d)} j\vspace{2mm}\\
				= & \sum_{a=1}^{M_B}\left(\frac{B}{2^{a+1}d}\right)^2+\Oc(1)\vspace{2mm}\\
				= & \frac{B^2}{d^2}\left(\sum_{a=0}^{M_B+1}\left(\frac{1}{4}\right)^a-\frac{5}{4}\right)+\Oc(\log_2 B)\vspace{2mm}\\
				= & \frac{B^2}{4d^2}\left(\frac{1}{3}-\frac{16}{3}\left(\frac{1}{4}\right)^{M_B+2}\right)+\Oc(\log_2 B),\vspace{2mm}\\
			\end{array}$$
		where $M_B=\lfloor \log_2 B\rfloor$. Since $(1/4)^{\log_2 B + 2}=1/(16B^2)$, we now have
			$$\sum_{\text{even}~q\leq B/d}f(q)=\frac{B^2}{12d^2}+\Oc(\log_2 B).$$
		With these two sums handled, we now have
			$$\begin{array}{rl}
				\sum_{n\leq B}\phi_2(n) = & \sum_{d\leq B}\mu(d)\left(\sum_{\text{odd}~q\leq B/d}f(q)+\sum_{\text{even}~q\leq B/d}f(q)\right).\vspace{2mm}\\
				= & \sum_{d\leq B}\mu(d)\left(\frac{B^2}{4d^2}+\frac{B^2}{12d^2}+\Oc(\log_2 B)\right)\vspace{2mm}\\
				= & \frac{B^2}{3}\sum_{d\leq B}\left(\frac{\mu(d)}{d^2}\right)+\Oc(\log_2 B)\sum_{d\leq B}\mu(d)\vspace{2mm}\\
				= & \frac{B^2}{3}\left(\frac{6}{\pi^2}+\Oc(\frac{1}{B})\right)+\Oc(B\log_2 B),\vspace{2mm}\\
			\end{array}$$
		using \cite[\S 3.7]{Apostol} and \cite[Theorem 3.13]{Apostol}, respectively, to compute the sums involving $\mu(d)$.
	\end{proof}
	\begin{theorem}
		\label{densityPPTg1}
		$\lim_{B\to\infty}\frac{\#\GcO(B)}{\#\Pc(B)}=\frac{1}{3}.$
	\end{theorem}
	\begin{proof}
		From the definition of $\GcO(B)$ and Proposition \ref{equalparity}, we have 
			$$\begin{array}{rl}
				\#\GcO(B) = & \#\braces{(k,m):~\gcd(k,m)=1,~0<m<k\leq B,~k,m\equiv 1\mod{2}}\vspace{2mm}\\
				= & \sum_{\text{odd}~k\leq B} \frac{1}{2}\phi(k)\vspace{2mm}\\
				= & \frac{1}{2}\sum_{k\leq B} \phi_2(k).\\
			\end{array}$$
		and thus from Theorem \ref{2Eulertheorem}, $\#\GcO(B)=\frac{1}{\pi^2}B^2+\Oc(B\log B).$ Combining this with Proposition \ref{asymptoticPPT} gives the result.
	\end{proof}
	\subsection{Case $g=2m^2$ even} In this case, we must consider both the case of $m$ even and of $m$ odd. Recall, in the comments following Proposition \ref{asymptoticPPT}, that in this case we compare the PPT of $(a,b,b+2m^2)$ form with all PPT of $(2rs,r^2-s^2,r^2+s^2)$ form.\\
	\indent For $m$ even, following with Corollary \ref{CorConj2}, we define
		$$\GcEE(B)=\braces{(k,m):~\gcd(k,m)=1,~0<m<k\leq B,~m\equiv 0,~k\equiv 1\mod{2}}\subset\Z^2.$$
	\begin{theorem}
		\label{densityPPTg2}
		$\lim_{B\to\infty}\frac{\#\GcEE(B)}{\#\Pc(B)}=\frac{1}{3}.$
	\end{theorem}
	\begin{proof}
		From the definition of $\GcEE(B)$ and Proposition \ref{equalparity}, we have 
			$$\begin{array}{rl}
				\#\GcEE(B) = & \#\braces{(k,m):~\gcd(k,m)=1,~0<m<k\leq B,~m\equiv 0,~k\equiv 1\mod{2}}\vspace{2mm}\\
						   = & \sum_{\text{odd}~k\leq B} \frac{1}{2}\phi(k)\vspace{2mm}\\
						   = & \frac{1}{2}\sum_{k\leq B} \phi_2(k).\\
			\end{array}$$
		and thus from Theorem \ref{2Eulertheorem}, $\#\GcEE(B)=\frac{1}{\pi^2}B^2+\Oc(B\log B).$ Combining this with Proposition \ref{asymptoticPPT} gives the result.
	\end{proof}
	\indent We now consider $m$ odd. As above, we compare the PPT of $(a,b,b+2m^2)$ form with all PPT of $(2rs,r^2-s^2,r^2+s^2)$ form. Following Corollary \ref{CorConj2}, we define
		$$\GcEO(B)=\braces{(k,m):~\gcd(k,m)=1,~0<m<k\leq B,~m\equiv 1,~k\equiv 0\mod{2}}\subset\Z^2.$$
	\begin{theorem}
			\label{densityPPTg3}
		$\lim_{B\to\infty}\frac{\#\GcEO(B)}{\#\Pc(B)}=\frac{1}{3}.$
	\end{theorem}
	\begin{proof}
		From the definition of $\GcEO(B)$, we have 
			$$\begin{array}{rl}
				\#\GcEO(B) = & \#\braces{(k,m):~\gcd(k,m)=1,~0<m<k\leq B,~m\equiv 1,~k\equiv 0\mod{2}}\vspace{2mm}\\
						   = & \sum_{\text{even}~k\leq B} \phi(k),\vspace{2mm}\\
			\end{array}$$
		as each $m$ with $\gcd(k,m)=1$ is necessarily odd when $k$ is even. Now, from 
			$$\sum_{k\leq B}\phi(k)=\sum_{\text{even}~k\leq B} \phi(k) + \sum_{\text{odd}~k\leq B}\phi(k)
								   =\sum_{\text{even}~k\leq B} \phi(k) + \sum_{\text{odd}~k\leq B}\phi_2(k),$$
		\cite[Theorem 3.7]{Apostol} (as in Proposition \ref{asymptoticPPT}), and Theorem \ref{2Eulertheorem}, we obtain
			$$\sum_{\text{even}~k\leq B} \phi(k)=\frac{3}{\pi^2}B^2-\frac{2}{\pi^2}B^2 + \Oc(B\log B)=\frac{1}{\pi^2}B^2+\Oc(B\log B).$$

\end{proof}
	\section{Future Work}
	\label{futurework}
		The case of $(a,a+f,c)$ PPT has certainly not been treated as thoroughly as the case of $(a,b,b+g)$ PPT, and so the focus of future work will be those $(a,a+f,c)$ triples. In particular, in \S\ref{genPPTf}, we have not established a precise relationship between PPT of $(a,a+f,c)$ form and PPT of $(r^2-s^2,2rs,r^2+s^2)$ form (or flipped parity), as in Corollary \ref{CorConj2}.\\
		\indent As a first step, if $a=r^2-s^2$ then we have
			$$r^2-s^2+f=a+f=b=2rs.$$
		It follows that
			$$-f=r^2-2rs-s^2=(r-s)^2-2s^2,$$
		and so $(r-s,s)$ are solutions to a similar Pell equation as considered in \S\ref{genPPTf}. Considering $a=2rs$ instead only changes the sign of $f$ in the Pell equation. The goal in the end is to be able to establish density results for PPT of $(a,a+f,c)$ form analogous to \S\ref{densityPPTg}.\\
		\indent Theorem \ref{Conj3} and Algorithm \ref{algoPPTf} do not provide any information about how the PPT produced compare for different choices of $m$ (allowing for negative values also) and $\uk$. In addition, we have given no attention to the $\pm$ sign that comes along with solutions to the classical negative Pell equation. It would be interesting to study how to construct $(a,a+f,c)$ PPT with minimal $a$, and possibly to determine a relation between values $f$ and corresponding minimal $a$.
%
%
%
%
%
%
%
%
%
\bibliographystyle{plain}
\bibliography{PPT}

\end{document}